\documentclass{article}

\usepackage{latexsym,amssymb,lastpage}
\usepackage{graphicx,amsfonts}
\usepackage{times,mathptmx,bm,amsmath}
\usepackage{dcolumn}
\usepackage{natbib}

\newtheorem{definition}{Definition}
\newtheorem{lema}{Lema}
\newtheorem{theorem}{Theorem}
\newtheorem{proof}{Proof}

\begin{document}
\title{Dynamic analysis in Greenberg's traffic model}
\maketitle
	
\begin{center}
Oscar A. Rosas--Jaimes\\
Facultad de Ciencias de la Electr\'onica, Benem\'erita Universidad Aut\'onoma de Puebla,\\
Prol. 24 Sur S/N Ciudad Universitaria, San Manuel, 72570 Puebla, México.\\
e-mail: oscar.rosasj@correo.buap.mx\\
\vspace{2em}
Luis A. Quezada--T\'{e}llez\\
Departamento de Matem\'aticas Aplicadas y Sistemas, UAM-Unidad Cuajimalpa,\\
Avenida Constituyentes 1054, 11950 Ciudad de M\'exico, M\'exico.\\
e-mail: lquezada@correo.cua.uam.mx\\
\vspace{2em}
Guillermo Fern\'andez--Anaya\\
Departmento de F\'{\i}sica y Matem\'aticas, Universidad Iberoamericana,\\
Prol. Paseo de la Reforma 880, \'Alvaro Obreg\'on, Lomas de Sta. Fe, Cd. de M\'exico,\\
e-mail: guillermo.fernandez@ibero.mx
\end{center}

\begin{abstract}
{Based on the classical traffic model by Greenberg, a linear differential equation, we analyze it by means of varying the critical velocity $v_o$ that appears in it as a parameter. In order to make such analysis we have obtained a solution for such a model and discretized it, obtaining related expressions for density $k$, flow $q$ and velocity $v$ to be treated as paired functions to obtained maps in phase-planes in which it is possible to observe distinct behaviors which span from monotonic and oscillatory stable trajectories, limit cycles of distinct periodicity, and chaotic ones. These behaviors are analyzed from a dynamical approach and then ilustrated with simulations performed in each case. As it is shown in this paper, these analyses are similar to those carried out in similar though simpler expressions (i.e. logistic-type functions), but taking in this case a new and direct approach through a nonlinear expression not used before to perform studies like these presented in this document, with a deep detail in the manner in which traffic variables are involved qualitatively and quantitatively.}

\end{abstract}

\section{Introduction}

Traffic of vehicles has a presence in every aspect of persons and products mobility that it has become a phenomenon by itself. The physics of traffic can be modeled by many approaches. Due to its probabilistic aspects, it can be seen as a set of stochastic models (see \cite{wang2013stochastic}), but it is undeniable that traffic observe rules and relations showing a deterministic character (see \cite{May90}). This ambiguous nature leads to adapt the best scheme depending on the planning, research or design needs.

Many traffic specialists prefer to follow more practical divisions that depend on the involved measures and variables. In this way, traffic can be explained by its individual units, the vehicles or small groups of them, through their positions, gaps, velocities or accelerations, obtaining microscopic models (see \cite{PanwaiDia2005}). These type of models have the advantage of focusing in fine details, making possible to distinguish among different classes of vehicles or even to simulate driving styles (see \cite{TreiberHelbing01}). However, these advantages are at the same time handicaps in other circumstances, due to they are difficult to analyze under a global perspective.

Most of the times, cumulative approaches are used in road networks because traffic needs to be emulated or analyzed as a continuum, using aggregated variables which are modeled, simulated, analyzed and controlled macroscopically as a stream of vehicles (see \cite{Lighthill1}, \cite{Lighthill2} and \cite{Richards56}). This point of view allows to manage a large set of road networks today, becoming the main reason of its extensive use by many professionals, from its first developments to nowadays (see for example \cite{MW2010}).

A first aspect of the nature of the traffic flow, independently of the type of model being used, is that it is nonlinear. As a consequence, all models related to traffic flow can exhibit a distinguishable behavior, ranging from stable and monotonous conditions to cyclic behaviours, which in turn can go from free-flow conditions to congested regimes, in the effort of emulating real situations (see \cite{Daganzo}).


Through the distinct types of models it is possible to perform analysis in order to help understand and predict those dynamical behaviours observed in roads. However, it is well known that there is no model able to describe the whole range of possibilities that can be observed in any traffic scenario. Some of such models are better than others for a certain kind of analysis, while others have a better fit to necessities of different understanding.

In Section \ref{sec:TM} a brief view related to traffic models is developed, but it focuses on those models that have a macroscopic approach and their variables. These models have a useful and graphical tool that helps in understanding their theoretical background, known as fundamental diagrams, which represents basic relationships between pairs of common macroscopic variables. This document is devoted to the analysis of a specific behavior of traffic through three possible fundamental diagrams, paying attention to some trajectories that can be observed when an adjustable parameter is modified.

In this same section a very well-known classical model proposed by \cite{Greenberg59} is used in this manuscript to generate a discrete version to perform analysis based on the fundamental diagrams generated directly from it. Fundamental diagrams have very well identified generic forms, but their exact shapes depend on the traffic model over which they are calculated. Even though it has been shown that Greenberg's model is not well suited for free-flow values of velocity, it provides a sufficiently and easy approach to emulate average values of velocity $v$, density $k$ and flow $q$ of a stream of vehicles moving in a road.

Section \ref{sec:MatConc} presents some mathematical concepts and definitions, as well as some propositions that will be useful for the dynamical analyses applied to this discretized model, while Section \ref{sec:Sims} presents some graphical results from numerical simulations obtained by varying a parameter that is capable to modify the scale of these fundamental diagrams, which is equivalent to the modification of the system and its conditions, as when the capacity permits smaller or bigger volume of vehicles, or the environmental conditions allow or restrict the driving possibilities. it is possible to identify points and ranges that are related to those analysis and their respective plotting representations, from stable fixed points to chaotic orbits.

Many documents (see \cite{ShihChingLo2005}, \cite{Low1998} and \cite{Villalobos2010}) present analyses of chaotic trajectories on traffic model systems, but they are related to microscopic approaches and not to macroscopic ones. \cite{ShihChingLo2005} develop an analysis over Greenshield's model, even though they adapt that model to a logistic-type mapping only for a flow--density relation. In \cite{RosasPROMET2016} a polynomial approach is suggested for the calculation of a flow--density fundamental diagram, and then it is used to perform a discrete dynamical analysis. In the present case of this article the velocity--density and the velocity--flow fundamental diagrams are also included, and they are not adapted versions of a fundamental diagram model, but they are directly derived from the solution of Greenberg's model and other basic expressions in traffic theory.

 This has resulted in a new discrete nonlinear expression not used before to achieve stable, cyclic and chaotic trajectories.


In Section \ref{sec:LyapExp}, further analysis is provided through measuring the chaos character, such as the divergence speed of the nearby trajectories using Lyapunov exponents, which are directly related with the nonlinear nature of the model used, being at the same time a qualitative indicator of stable or chaotic behavior of the system.

At the end, some concluding remarks are written for the main results and consequences obtained.

\section{\label{sec:TM}Traffic Models}

\subsection{\label{FD_Sec}Fundamental Diagrams}
Traffic models arise mainly from a necessity of understanding phenomena that have strong implications in economic and social aspects of modern life, specially when the physics of such phenomena has to do with an increasing frequency congestion (see \cite{Beckmann2013}). First attempts to produce them began with statistical approaches as well as analogies with other physical phenomena (see \cite{Lighthill1}, \cite{Lucic2002} and \cite{Wohl}), with the intention they match up with real data, which seem to follow curves that began to be called {\em fundamental diagrams}.

\begin{figure}[ht]
\centering
\includegraphics[width=1.0\textwidth]{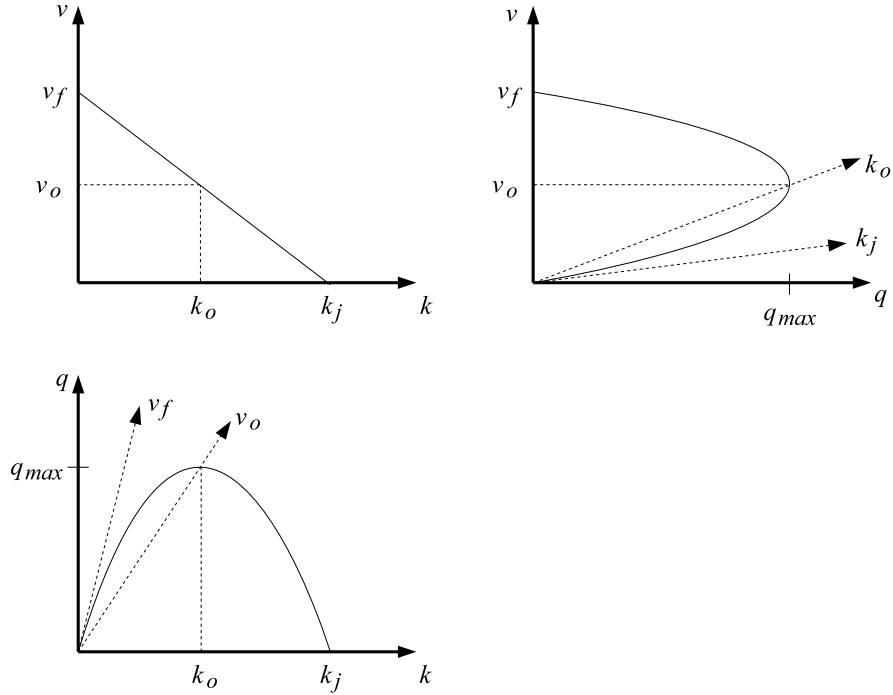}
\caption{Plots of fundamental relations among flow ($q$), velocity ($v$) and density ($k$)}
\label{DiagFundCurves}
\end{figure}

Figure \ref{DiagFundCurves} shows the flow-density map, the velocity-density map and the velocity-flow map for a linear velocity-density relation model, from where the other two relationships can be obtained (see \cite{May90}). These functions are idealizations of the true points that represent the values of the traffic macroscopic variables, and their utility lies in that they are the best comprehensive approaches for any set of traffic macroscopic plots.

In a macroscopic approach to traffic flow, velocity $v$ is referred as the velocity of the wave front of vehicles, many times considered as the maximum velocity reached by the average of the cars. When congestion appears, it can be more valuable to know the backward front of congestion velocity $w$, a quantity that measures a congestion wave front that moves in the opposite direction to traffic flow (see \cite{kerner2012physics}). However, few works have taken advantage of this quantity (see for example \cite{RosasJDS07}).

Flows $q$ are quantities related to a set of vehicles moving with respect to time. It can be a number of cars trespassing a point in a
period of time or a set of cars traversing a section of a way in a time interval (see \cite{ITE2009}).

Density $k$ is defined as the number of vehicles occupying a section of a lane or stretch of a road. Direct measure of density can be get by air photographs, videorecorded images or by \textit{in situ} observations, limiting a length of a way and counting the vehicles
present on it in a moment of time (see \cite{Thamizh2010}).

These three quantities are simplistic related through
\begin{equation}
q=vk \label{qvk}
\end{equation}
and relations between pairs of them are represented in Figure \ref{DiagFundCurves}. The form of these curves is rather descriptive and at the same time idealized. It depends on particular cases of particular roads and their conditions. Data sets of each of them depict a complete and continue function but it is little probable to find the whole range of values of each variable in a measuring location. Data obtained in real life have multiple discontinuities in which many parts of these curves are not present (see \cite{KimZhang}).

Nevertheless these curves illustrates several significative points. Note that null flow occurs in two different conditions:
\begin{enumerate}
\item When there are no cars in the road, density and flow are zero. Velocity is theoretical and it will be that of the first driver appearing, supposedly a high value. This velocity is represented in the fundamental diagram as $v_f$ and it is named {\em free-flow velocity}.
\item When density becomes so high that all vehicles are forced to stop, flow is zero again, due to there is no movement. The density in this situation is known as {\em jam density} and it is referred to as $k_j$.
\end{enumerate}

Between these two extremes there are many conditions of vehicular flow. As density increments from zero, flow does the same, due to there are more cars on the road, but as this happens velocity declines, because of the growing interaction among vehicles. This decrement in velocity is imperceptible when densities and flows are low.

Velocity decreases significatively a little before reaching the maximum flow. This condition is showed in the diagrams of Figure \ref{DiagFundCurves} as the \textit{critic velocity} or \textit{optimum velocity} $v_o$, the \textit{optimum density} $k_o$ and the \textit{maximum flow} $q_{max}$.

Slope of a straight line drawn from the origin of the velocity-flow diagram towards any point in the curve represents density. Likewise, a straight line from the origin of the density-flow diagram to any point over the curve represents velocity $v$. These slopes can be calculated from Equation (\ref{qvk}). Another important slope is that corresponding with the backward front velocity of congestion $w$, not depicted in the figure, but that would be plotted from the jam density extreme and with a negative slope.

Notice that the three diagrams showed are redundant, because once established a relation between two of the variables, the another stays defined. Each of these relations (and their respective fundamental diagram) has its own field of application. For example, the velocity-density diagram is the base for vehicular flow models, because for a density value corresponds only a velocity value. The flow-density relationship is the point of departure for traffic control, due to it is possible to identify easily those regions where the traffic can be consider free, congested or in transition. The velocity-flow relationship is useful because it depicts regions of these values that can be related directly with levels of service in the roads (see \cite{ITE2009} and \cite{May90}).

As can be observed in Figure \ref{DiagFundCurves}, any flow value distinct of the maximum can occur in two different conditions, one with low density and high velocity, and another with high density but low velocity. The portion of the curves for this last case represents the congested situation, with sudden changes in the traffic, some of them periodic, some of them chaotic, as will be seen later.

\subsection{\label{GM}Greenberg's Model}
Equation (\ref{qvk}) relates density, velocity and flow, but it tells little about the way in which the corresponding data matches with the model used, and several fitness schemes has been proposed. Unfortunately, traffic flow data seem to be quite complex, and no model is able to achieve a perfect fit. Some models are better than others for different traffic flow regimes.

Greenberg's model is a well stablished expression for macroscopic traffic, the result of observing velocity-density data sets for tunnels, specially those data that describe congestion (see \cite{Greenberg59}). This author correlated this information with the hydrodynamic analogy (given by the set of works by \cite{Lighthill1}, \cite{Lighthill2} and \cite{Richards56}) due to the equation of motion of a one-dimensional fluid
\begin{equation}
\frac{dv}{dt}=-\frac{v_o^2}{k} \frac{\partial k}{\partial x}
\label{OneDimF}
\end{equation}
where $x$ is the distance coordinate along the road, $t$ is time, the parameter $v_o$ is the optimum or critical velocity, and $v$ and $k$ are the already-known variables for velocity and density.

Due to velocity is a function of distance and time, and by the properties of the total derivative, then Equation (\ref{OneDimF}) can be written in the form
\begin{equation}
\frac{dv}{dk}\frac{\partial k}{\partial t} +
v\frac{dv}{dk}\frac{\partial k}{\partial x} +
\frac{v_o^2}{k}\frac{\partial k}{\partial x} = 0
\label{OneDimF2}
\end{equation}

which can be divided by $dv/dk$ to obtain
\begin{equation}
\frac{\partial k}{\partial t} + \left( v +
\frac{v_o^2}{k}\frac{dk}{dv}\right) \frac{\partial k}{\partial x} =
0 \label{OneDimF3}
\end{equation}

Greenberg then uses the mass conservation expression
\begin{equation}
\frac{\partial k}{\partial t}+\frac{\partial q}{\partial x}=0
\label{continuity}
\end{equation}

which, by means of Equation (\ref{qvk}), can be written as
\begin{equation}
\frac{\partial k}{\partial t} + v\frac{\partial k}{\partial x} +
k\frac{dv}{dk}\frac{\partial k}{\partial x} = 0
\label{continuity2}
\end{equation}

In this way, it is possible to describe the behavior of vehicles using Equations (\ref{OneDimF3}) and (\ref{continuity2}).

As this pair of equations constitutes a system, a non-trivial solution is achieved by this author as
\begin{equation}
\frac{dv}{dk} = - \frac{v_o}{k} \label{GreenbSol}
\end{equation}

Differential equation (\ref{GreenbSol}) follows the trajectory given by
\begin{equation}
\label{GreenbExp}
v = v_o \ln \left( \frac{k_j}{k} \right)
\end{equation}

which functional plot is shown in Figure \ref{FGreenbExpqvk} in the portion that corresponds to the velocity-density fundamental diagram where, as can be seen, infinity values for free-velocity are calculated. Even though this fact can be seen as a drawback for this model, Greenberg was able to show that for regions going away from those free-velocity values the model adjusts quite well.

\begin{figure}[ht]
\centering
\includegraphics[width=1.0\textwidth]{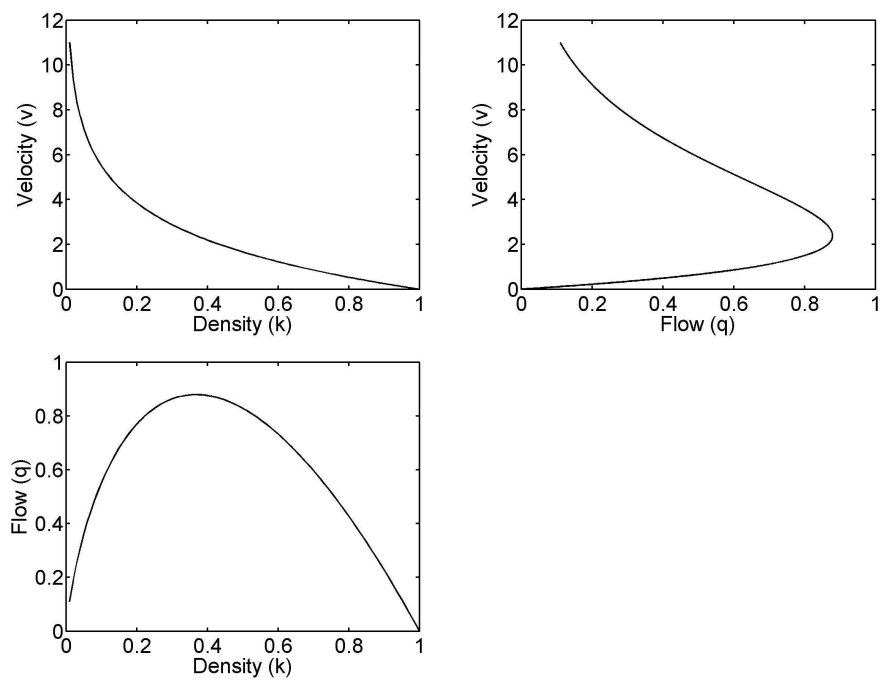}
\caption{Fundamental Diagrams based on Greenberg's Traffic Model with normalized values} \label{FGreenbExpqvk}
\end{figure}

Combining equations (\ref{qvk}) and (\ref{GreenbExp}) results in a relationship for density $k$ and flow $q$
\begin{equation}
\label{GreenbExpqk}
q = v_o \,k \ln \left( \frac{k_j}{k} \right)
\end{equation}
which is depicted in Figure \ref{FGreenbExpqvk} in the respective portion of the flow-density fundamental diagram.

The corresponding plot for the velocity-flow relationship is easier to get through a set of values obtained from flow and density data
through the Equation (\ref{qvk}) in the form
\begin{equation}
k=\frac{q}{v} \label{vqk}
\end{equation}
This one is represented in its respective portion of Figure \ref{FGreenbExpqvk}. This figure gives a deeper detail of the relationship of the variables density $k$, flow $q$ and velocity $v$ in comparison to Figure 1. In this way, an improved accuracy for the fundamental diagrams used for the analysis is shown by Figure 2. 

\cite{Greenberg59} calculated and analyzed parameter values such as the maximum flow $q_{max}$, the optimum velocity $v_o$ and the jam density $k_j$. These last two expressions are obtained throughout direct observations of graphical representations of the data.

Real collections of traffic data show big variance and dispersion and, as has been mentioned earlier, it is difficult for any model to
achieve a satisfactory fit to them. \cite{Greenberg59} commented on this fact, pointing out that his model represents a very good average of any set of traffic data for a wide range of road conditions, what has been confirmed by latter measurements (see \cite{May90}).

Let the set of equations (\ref{GreenbExp}), (\ref{GreenbExpqk}) and (\ref{vqk}) be the functions which describe the relationship between each pair of the macroscopic variables of traffic. In order to generalize a result, let the maximum density $k_j=1$. with this assumption and from the flow-density function, it is possible to notice that the maximum flow $q_{max}$ occurs when $k_j/k = e$, and the optimum density is then $k_o=e^{-1}$.

With these values, a normalization can be established for this set of functions, allowing to lead interesting analyses about some generalized properties of those expressions related to the general shape of the fundamental diagrams obtained directly by the Greenberg's model.

\begin{figure}[ht]
\centering
\includegraphics[width=0.9\textwidth]{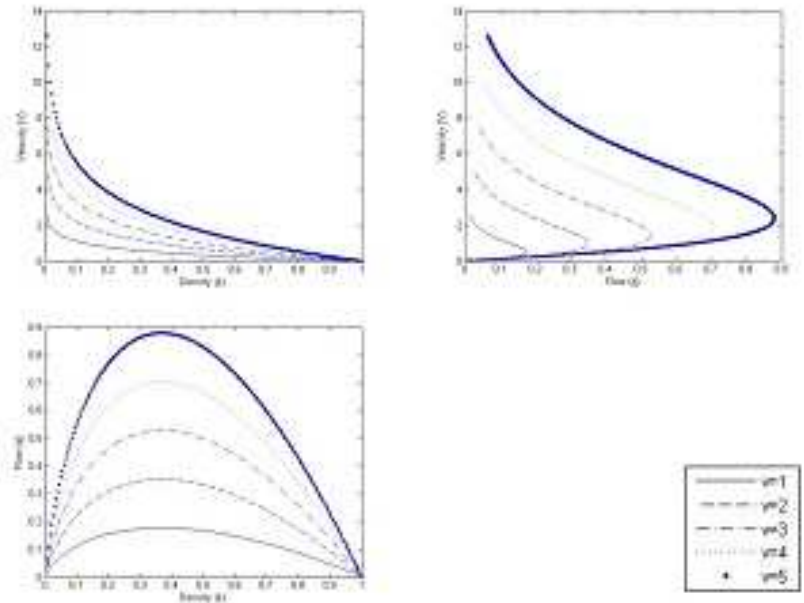}
\caption{Fundamental Diagrams based on Greenberg's Traffic Model showing the variation of their shapes through five different values of $v_o$} \label{fig:Greenb5vo}
\end{figure}

These analyses will be conducted through the optimum velocity $v_o$, a parameter that can be used to adjust these fundamental diagrams to different type of traffic data, with the effect of enlarging or shrinking their shapes. Figure \ref{fig:Greenb5vo} shows the aspect obtained by the fundamental diagrams for five distinct values of $v_o$. Physically, this is equivalent to modify the road conditions for the model, since for a bigger value of $v_o$ corresponds a bigger capacity or some other enhanced characteristic for vehicles displacement.

The parameter $v_o$ is left constant for each specific process and modifies the functions (\ref{GreenbExp}) and (\ref{GreenbExpqk}) of variable density $k$. These functions, along with that of Equation (\ref{vqk}) preserve the shape of the respective fundamental diagrams as seen in Figure \ref{fig:Greenb5vo}. By the normalization made, it is possible to see that $k \in \left[ 0\,1 \right]$ and $q \in \left[ 0\,1 \right]$.

We want to treat these fundamental diagrams in a similar manner as is done in the paper by \cite{RosasPROMET2016}. In that paper, only the flow--density fundamental diagram is tested by means of a polynomial expression. In this present work we are spanning such treatment to those three fundamental diagrams already mentioned, using Greenberg's model. In consequence, each of these three fundamental diagrams are supposed to achieve an independent and parallel iterative process of the form $k_{i+1} = Q\left( k_i \right)$, where $Q\left( k_i \right)$ is the particular function (\ref{GreenbExp}) or (\ref{GreenbExpqk}).


\begin{equation}
Q_v(k_{i + 1}) = v_o \ln \left( \frac{k_j}{k(i)} \right),\,\, i=1,2,3,...
\label{yGreenbExp}
\end{equation}
\begin{equation}
Q_q(k_{i+1}) = v_o \, k(i) \ln \left( \frac{k_j}{k(i)} \right),\,\, i=1,2,3,...
\label{kiter}
\end{equation}

That is to say, starting from an initial condition, these expressions will generate a new value for $Q_v(k_{i+1})$ and $Q_q(k_{i+1})$ from previous inputs $Q_v(k_i)$ and $Q_q(k_i)$, constructing sequences $\left\{ Q_v(k_n) \right\} = \left\{Q_v(k_0)\, Q_v(k_1), \cdots , Q_v(k_i), \cdots \right\}$ and $\left\{ Q_q(k_n) \right\} = \left\{Q_q(k_0)\, Q_q(k_1), \cdots , Q_q(k_i), \cdots \right\}$ in an iterative way, and all this values will create a set that can be registered and plotted. These processes have been performed in other works over logistic-type functions (see, for example \cite{Thamizh2010}, \cite{Devaney1987}, \cite{Holmgren1994} and \cite{ShihChingLo2005}).

As it was mentioned earlier, one of the main purposes of this work is to fit the shapes obtained from Greenberg's model to calculate the same iterative processes that appear in those published papers, where such iterations are done over one-dimension functions of the form $X_{i+1} = F(X_i)$. Even though Equations (\ref{GreenbExp}) and (\ref{GreenbExpqk}) are evidently two-dimension functions, due to the normalization properties, the fundamental diagrams obtained can follow the same processes.

Physically, this would be the case of a closed traffic network composed by arcs and nodes, all included in a control volume, but only a point of measuring is taken into account to know the state, and interactions and geometry of this network will be reflected in that sole point.

In order to complete our three-variable scheme, density $k$ is calculated in a parallel form by
\begin{equation}
k(i+1) = \frac{Q_q(k_i)}{Q_v(k_i)}
\label{kiter1}
\end{equation}


\section{\label{sec:MatConc}Mathematical Concepts}
In this section, some useful definitions, theorems and lemmas are presented, which will be used for the analyses that will be developed through the rest of this article.

\subsection{Fixed Points in Greenberg's Discrete Model}
\vspace{1em}

\begin{definition}\label{DefDiscreteSyst}
	(\cite{Luo2012}) For $\Omega_{\alpha} \subseteq \mathbb{R}^{n}$ and $\Lambda \subseteq 				\mathbb{R}^{m}$ which $\alpha \in \mathbb{Z}$, consider a vector function $\mathbf{f}_{\alpha}: \Omega_{\alpha} \times 			\Lambda \rightarrow \Omega_{\alpha}$ wich is $C^r(r\geq 1)$-continuous, and there is a discrete equation in the form of
		\begin{equation}
		\mathbf{k}_{i+1}=\mathbf{f_\alpha}(\mathbf{k}_{i},\mathbf{p_\alpha}) 
		\label{kiterf}
		\end{equation}
	for $\mathbf{k}_{i},\mathbf{k}_{i+1} \in \Omega_{\alpha}, i \in \mathbb{Z}$ and $\mathbf{p}_{\alpha} \in \Lambda$. With 			an initial condition of $\mathbf{k}_{i}=\mathbf{k}_{0}$, the solution of equation (\ref{kiterf}) is given by
		\begin{equation}
		\mathbf{k}_{i}=\mathbf{f}_{\alpha}(\mathbf{f}_{\alpha}(\cdots(\mathbf{f}_{\alpha}(\mathbf{k}_{0},						\mathbf{p}_{\alpha}))))
		\end{equation}
	for $\mathbf{k}_{i} \in \Omega_{\alpha}, i \in \mathbb{Z}$ and $\mathbf{p} \in \Lambda$.
		\begin{enumerate}
		\item The difference equation with the initial condition is called a discrete dynamical system.
		\item The vector function $f_\alpha(k_i, p_\alpha)$ is called a discrete vector field on domain $\Omega_\alpha$.
		\item the solution $k_i$ for all $i \in \mathbb{Z}$ on domain $\Omega_\alpha$ is called the								\text{\normalfont{trajectory}}, \text{\normalfont{phase curve}} or \text{\normalfont{orbit}} of discrete dynamical 				system, which is defined as
			\begin{equation}
			\Gamma = \left\{ k_i | k_{i+1} = \mathbf{f_\alpha}(\mathbf{k}_{i},\mathbf{p_\alpha}) \mbox{ for } k \in 						\mathbb{Z} \mbox{ and } \mathbf{p}_{\alpha} \in \Lambda \right\} \subseteq \cup_\alpha\Omega_\alpha
			\end{equation}
		\end{enumerate}
\end{definition}
\vspace{1em}

\begin{definition}\label{DefFixPoint}
	(\cite{Luo2012}) Consider a discrete, nonlinear dynamical system $\mathbf{k}_{i+1}=\mathbf{f}				(\mathbf{k}_{i},\mathbf{p})$. A point $\mathbf{k}^{*}_{i} \in \Omega_{\alpha}$ is called a \text{\normalfont{fixed point}} 			or a \text{\normalfont{period-1}} solution of a discrete nonlinear system $\mathbf{k}_{i+1}=\mathbf{f}(\mathbf{k}_{i},			\mathbf{p})$ under map $P_{i}$ if for $\mathbf{k}_{i+1}=\mathbf{k}_{i}=\mathbf{k}^{*}_{i}$
		\begin{equation}
		\mathbf{k}^{*}_{i}=\mathbf{f}(\mathbf{k}^{*}_{i},\mathbf{p})
		\end{equation}
	The linearized system of the nonlinear discrete system $\mathbf{k}_{i+1}=\mathbf{f}(\mathbf{k}_{i},\mathbf{p})$ at the 			fixed point $\mathbf{k}^{*}_{i}$ is given by
		\begin{equation}
		\mathbf{y}_{i+1}=DP(\mathbf{k}^{*}_{i},\mathbf{p})\mathbf{y}_{i}=D\mathbf{f}(\mathbf{k}^{*}_{i},						\mathbf{p})\mathbf{y}_{i}
		\end{equation}
	where,
		\begin{equation}
		\mathbf{y}_{i}=\mathbf{k}_{i} - \mathbf{k}^{*}_{i} \ and \ \mathbf{y}_{i+1}=\mathbf{k}_{i+1} - 						\mathbf{k}^{*}_{i+1}
		\end{equation}
\end{definition}

By Definition \ref{DefFixPoint} and if $\mathbf{k}_{j}=1$ as mentioned in Section \ref{sec:TM}, a fixed point in Greenberg's model is
\begin{equation}
\mathbf{k}^{*}_{i} = e^{-\frac{1}{v_{o}}}
\label{GreenbergFixPoint}
\end{equation}

Derivation of (\ref{kiter}) gives
\begin{equation}
\frac{d}{dk_{i}}\left[v_{o}\ \mathbf{k}_{i}\ ln\left(\frac{\mathbf{k}_{j}}{\mathbf{k}_{i}}\right)\right] = -v_{o} + v_{o}\left(ln\mathbf{k}_{j} - ln\mathbf{k}_{i}\right)
\end{equation}

Evaluating this derivative with $k_j=1$
\begin{equation}
\frac{d}{dk_{i}}\left[v_{o}\ \mathbf{k}_{i}\ ln\left(\frac{\mathbf{k}_{j}}{\mathbf{k}_{i}}\right)\right] \bigg|_{\mathbf{k}^*_{i}} = -v_{o}(1-\frac{1}{v_{o}}) = -v_{o} + 1 
\end{equation}

\begin{definition}\label{FixPointTypes}
	(\cite{Luo2012})  Consider a discrete nonlinear dynamical system $\mathbf{k}_{i+1} = \mathbf{f}				(\mathbf{k}_{i},\mathbf{p})$ with a fixed point $\mathbf{k}^*_{i}$. The corresponding solution is given by 					$\mathbf{k}_{i+j} = \mathbf{f}(\mathbf{k}_{i+j-1},\mathbf{p})$ with $j\in\mathbb{Z}$. Suppose there is a neighborhood of 		the fixed point $\mathbf{k}^*_{i}(i.e.,\hspace{0.001cm} U_{i}(\mathbf{k}^*_{i})\subset\Omega_{\alpha})$, and 				$\mathbf{f}(\mathbf{k}_{i},\mathbf{p})$ is $C^r(r\geq1)$-continuous in $U_{i}(\mathbf{k}^*_{i})$. The linearized system is 		$\mathbf{y}_{i+j+1} = D\mathbf{f}(\mathbf{k}^*_{i},\mathbf{p})\mathbf{y}_{i+j}(\mathbf{y}_{i+j}=\mathbf{k}_{i+j} - 			\mathbf{k}^*_{i})$ in $U_{i}(\mathbf{k}^*_{i})$. The matrix $D\mathbf{f}(\mathbf{k}^*_{i},\mathbf{p})$ possesses $n$ 			eigenvalues $\lambda_{q}\hspace{1em}(q=1,2,...,n)$.
		\begin{enumerate}
		\item The fixed point $\mathbf{k}^*_{i}$ is called a hyperbolic point if $\lvert \lambda_{q} \rvert 	\neq 1$ \hspace{1em}$(q=1,2,...,n)$.
		\item The fixed point $\mathbf{k}^*_{i}$ is called a sink if  $\lvert \lambda_{q} \rvert < 1$\hspace{1em}$(q=1,2,...,n)$.
		\item The fixed point $\mathbf{k}^*_{i}$ is called a source if $\lvert \lambda_{q} \rvert > 1$\hspace{1em}$(q=1,2,...,n)$.
		\item The fixed point $\mathbf{k}^*_{i}$ is called a center if $\lvert \lambda_{q} \rvert = 1$\hspace{1em}$(q=1,2,...,n)$ with distinct eigenvalues.
		\end{enumerate}
\end{definition} 
\vspace{1em}

We apply these definitions to the particular case of the fixed point found for Greenberg's model. Notice that $v_0$ is always a non-negative value. Therefore:

\begin{enumerate}
	\item For  $\lvert-v_{o} + 1\rvert = 1$ only when $v_0 = 0$. Except for such a value $\mathbf{k}^*_{i}$ is a hyperbolic 			point.
	\item For $v_{0} \in \left(0,2\right)$, we have $\lvert-v_{o} + 1\rvert < 1$, then $\mathbf{k}^*_{i}$ is a sink.
	\item For $v_{0} \in (-\infty,0) \cup (2,\infty)$, it is obtained $\lvert-v_{o} + 1\rvert > 1$ and then $\mathbf{k}^*_{i}$ is a 			source.
	\item For values $v_{0}=0$ and $v_{0}=2$, $\lvert-v_{o} + 1\rvert = 1$ and in such cases a pair of centers result, a 				hyperbolic and a non-hyperbolic.
	\end{enumerate}
\vspace{1em}

\subsection{Bifurcation in Greenberg's Discrete Model }
\vspace{1em}

\begin{definition}\label{bifurcation}
	\cite{Luo2012} Consider a 1-D map
		\begin{equation}
		\label{mDmap}
		P: \mathbf{k}_{i} \rightarrow \mathbf{k}_{i+1} \hspace{0.5cm} with \hspace{0.5cm} \mathbf{k}_{i+1}=\mathbf{f}				(\mathbf{k}_{i},\mathbf{p})
		\end{equation}
	where $\mathbf{p}$ is a parameter vector. To determine the period-$1$ solution (fixed point) of Equation \ref{mDmap}, substitution of $k_{i+1}=k_{i}$ into Equation \ref{mDmap} yields the periodic solution $k_{i}=k^*_{i}$. The bifurcation of the period-$1$ solution is presented. 
		\begin{enumerate}
		\item Period-doubling bifurcation
			\begin{equation}
			\frac{dk_{i+1}}{dk_{i}}=\frac{df(k_{i},\mathbf{p})}{dk_{i}}\bigg|_{\mathbf{k}_{i}=\mathbf{k}^*_{i}}=-1
			\label{PitchBifur}
			\end{equation}
		\end{enumerate}
\end{definition}

As it will be depicted in Section \ref{sec:Sims}, Greenberg's model (\ref{vqk})--(\ref{kiter}) present $n$-period cycles in its trajectories and bifurcations in corresponding mappings. Taking $n=1$ in (\ref{mDmap}) a period-doubling bifurcation presents when $v_{0} = 2$, through Definitions \ref{FixPointTypes} and \ref{bifurcation}.

\begin{equation}
\frac{df(k_{i},\mathbf{p})}{dk_{i}}\bigg|_{\mathbf{k}_{i}=\mathbf{k}^*_{i}}=-v_{o} + 1=-1
\end{equation}

And then
\begin{equation}
v_{0} = 2
\end{equation}

\subsection{Stability analysis for Greenberg's Discrete Dynamic System}
In this last subsection, we include the fundamentals of stability for a nonlinear discrete dynamical system (see \cite{Ngoc2012}). Consider a nonlinear discrete-time system of the form 

\begin{equation}
\mathbf{k}_{i+1} = \mathbf{f}(i,\mathbf{k}_{i}), i \geq i_{0}
\end{equation}

where $\mathbf{f}: \mathbb{Z_{+}} \times \mathbb{R}^{n} \rightarrow \mathbb{R}^{n} $ is a given function such that $\mathbf{f}(i,0) = 0$, for all $i \in \mathbb{Z_{+}}$ (i.e., $\xi = 0$ is an equilibrium of the system (\ref{kiter})). 
It is clear that for $i_{0} \in \mathbb{Z_{+}}$ and $\mathbf{k}_{0} \in \mathbb{R}^{n}$, (\ref{kiter}) has a unique solution, denoted by $\mathbf{k}_{i}(\centerdot,i_{0},\mathbf{k}_{0})$ satisfaying the initial condition 
\begin{equation}
\mathbf{k}_{i=0} = \mathbf{k}_{0}
\end{equation}

\begin{definition}\label{ExpStable}
(\cite{Ngoc2012}) The \text{\normalfont{zero solution}} of (\ref{kiter}) is exponentially stable if there exist $M \geq 0$ and $\beta \in [0,1)$ such that:
	\begin{equation}
	\label{ExpCond}
	\forall i,i_{0} \in \mathbb{Z_{+}}, i \geq i_{0}; \forall \hspace{0.05cm} \mathbf{k}_{0} \in \mathbb{R}^n: \hspace{0.05cm}   			\parallel \mathbf{k}(i,i_{0},\mathbf{k}_{0})\parallel \leqslant M\beta^{i-i_{0}}\parallel \mathbf{k}_{0}\parallel
	\end{equation}
\end{definition}

A simple sufficient condition for exponential stability to (\ref{kiter}) is given by the following lemma \ref{lema1}. 

\begin{lema}\label{lema1}
(\cite{Ngoc2012}) Suppose there exists $ A \in \mathbb{R}^{n \times n}_{+} $ such that $ \lvert \mathbf{f}(i,\mathbf{k}_{i}) \rvert \leq A \lvert \mathbf{k}_{i} \rvert $, \hspace{0.01cm} $ \forall i \in \mathbb{Z_{+}}$, $ \forall \hspace{0.05cm} \mathbf{k}_{i} \in \mathbb{R}^n $. If $\rho(A) < 1 $ then the zero solution of (\ref{kiter}) is exponentially stable.
\end{lema} 

Direct application to (\ref{kiter}) is given in the following theorem.

\begin{theorem}\label{th1}
Consider the scalar system $\mathbf{k}_{i+1} = v_{0} \mathbf{k}_{i} \ln\left(\dfrac{\mathbf{k}_{j}}{\mathbf{k}_{i}}\right)$. If $\lvert v_{0}\rvert < 1$ and $\lvert \mathbf{k}_{i} \lvert < 1, \forall i \in \mathbb{Z}_{+}$  then the system is exponentially stable.  
\end{theorem}

\begin{proof}\label{proof1}
	From equation (\ref{kiter})
	\begin{equation}
	\mathbf{k}_{i+1} = v_{0}\mathbf{k}_{i}ln \left(\frac{\mathbf{k}_{j}}{\mathbf{k}_{i}}\right)
	\end{equation}

	let $\mathbf{k}_{j} = 1$
	\begin{equation}
	\mathbf{k}_{i+1} = v_{0}\mathbf{k}_{i}ln\mathbf{k}_{i}
	\end{equation}

	and taking absolute value 
	\begin{equation}
	\lvert \mathbf{k}_{i+1} \rvert = \lvert v_{0} \rvert \lvert \mathbf{k}_{i} \rvert \lvert ln\mathbf{k}_{i} \rvert
	\end{equation}

	we obtain 
	\begin{equation}
	\frac{\lvert \mathbf{k}_{i+1} \rvert}{\lvert \mathbf{k}_{i} \rvert}  = \lvert v_{0} \rvert \lvert ln\mathbf{k}_{i} \rvert < \lvert 		v_{0} \rvert \lvert \mathbf{k}_{i} \rvert 
	\end{equation}

	and rewriting 
	\begin{equation}
	\lvert \mathbf{k}_{i+1} \rvert < \lvert v_{0} \rvert \lvert \mathbf{k}_{i} \rvert^{2} 
	\end{equation}

	by hypothesis $ \lvert \mathbf{k}_{i} \rvert < 1, \forall i \in \mathbb{Z}_{+}$, in consequence

	\begin{equation}
	\lvert \mathbf{k}_{i+1} \rvert <  \lvert v_{0} \rvert \lvert \mathbf{k}_{i} \rvert^{2} \leq \lvert v_{0} \rvert \lvert 				\mathbf{k}_{i} \rvert 
	\end{equation}

	\begin{equation}
	\lvert \mathbf{k}_{i+1} \rvert < \lvert v_{0} \rvert \lvert \mathbf{k}_{i} \rvert
	\end{equation}
	\\
	therefore by Lemma \ref{lema1} if $\lvert v_{0} \rvert < 1$, then the system is exponentially stable.
\end{proof}
\vspace{1em}

These results about stability can be extended to the other two fundamental diagrams due to their relationships by means of expressions such as (\ref{yGreenbExp}) and (\ref{kiter1}). On the other hand, the implications in traffic situations from this theorem are directly related with decreasing or increasing amounts of traffic density $k$ and flow $q$, as will be better explained by the simulations performed and depicted in Section \ref{sec:Sims}.

The nonlinear nature of these equations leads to a quasi--periodic behaviour which reaches chaotic dynamics, whose trajectories are characterized by an exponential divergence of initially close points (as explained by \cite{Korsch1}). Instead of developing a stability analysis for such situations, we opt to take care about other type of measuring. Taking the case of one-dimensional discrete maps of an interval

\begin{equation}
k_{n+1}=f(k_{n}), \hspace{1cm} x \in [0,1]
\end{equation}

The so-called \textit{Lyapunov exponent} is a measure of the divergence of two orbits starting with slightly different initial conditions $k_{0}$ and $k_{0}+\Delta k_{0}$. The distance after $n$ iterations

\begin{equation}
\Delta k_{n} = \lvert f^{n}(k_{0}+\Delta k_{0})-f^{n}(k_{0}) \rvert
\end{equation}
increases exponentially for large $n$ for a chaotic orbit according to 

\begin{equation}
\Delta k_{n}\thickapprox \Delta k_{0}\,{\rm e}^\lambda_{L}
\end{equation}

It is possible to relate the Lyapunov exponent analytically to the average stretching along the orbit $k_{0},k_{1}=f(k_{0}),k_{2}=f(f(k_{0})),\ldots,k_{n}=f^{n}(k_{0})=f(f(f(\ldots(k_{0})\ldots))$. Through proper mathematical treament, it is possible to obtain:

\begin{equation}
\label{LogProd}
\ln \frac{\Delta k_{n}}{\Delta k_{0}} = \ln \left\vert \frac{f^{n}(k_{0}+\Delta k_{0})-f^{n}(k_{0})}{\Delta k_{0}} \right\vert \, \thickapprox \, \ln \left\vert   \frac{df^{n}(k)}{dk} \right\vert = \ln \prod_{j=0}^{n-1} \lvert f'(k_{j}) \rvert
\end{equation}
or equivalently

\begin{equation}
\label{LogSum}
\ln \frac{\Delta k_{n}}{\Delta k_{0}} = \sum_{j=0}^{n-1}\ln \lvert f'(k_{j}) \rvert
\end{equation}
and finally

\begin{equation}
\label{LyapExpCalc}
\lambda_{L} = \lim_{n\rightarrow \infty}\frac{1}{n}\ln \frac{\Delta k_{n}}{\Delta k_{0}} = \lim_{n\rightarrow \infty}\frac{1}{n}\sum_{j=0}^{n-1}\ln \lvert f'(k_{j}) \rvert
\end{equation}
where the logarithm of the linearized map is averaged over the orbit as $k_{0}$, $k_{1}$, $\ldots\,$, $k_{n-1}$. Negative values of the Lyapunov exponent indicate stability, and positive values chaotic evolution, where $\lambda_{L}$ measures the speed of exponential divergence of neighboring trajectories. At critical bifurcation points the Lyapunov exponent is zero (see \cite{Korsch1}).

\begin{figure}
\centering
\includegraphics[width=1.0\textwidth]{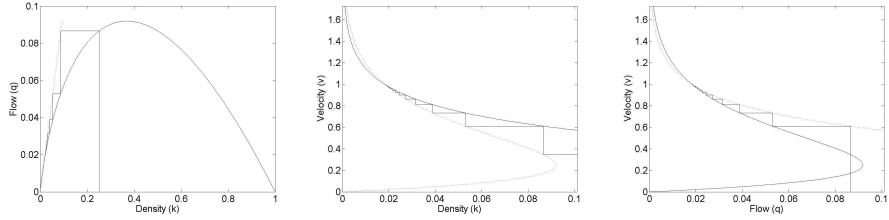}
\caption{Iterative behavior of the normalized fundamental diagrams. Optimum velocity $v_o=0.25$, initial condition $k=0.25$} \label{uo0p25}
\end{figure}

\section{\label{sec:Sims}Numerical Simulations}

As it has been shown through this document, discrete iterative equations have been defined for pairs of variables depicting fundamental diagrams density versus flow ($k$ vs $q$), density versus velocity ($k$ vs $v$) and flow versus velocity ($q$ vs $v$). These expressions are useful to perform simulations which plottings show interesting trajectories as parameter $v_0$ is varied. Each simulation is a set of 300 iterations.

If $k = 0.25$ is an initial condition with $v_{o}= 0.25$, a first series of iterations can be executed (Figure \ref{uo0p25}). Its possible to follow the trajectory of the iterations. The $q = k$ line is drawn over the flow-density plot, which is the most common depicted in literature. It is helpful in order to visualize the form in which iterations develop, represented by vertical and horizontal lines that intersect it from values calculated.

In that way, it can be noticed that starting from the initial condition the final value stops in a point near $q = k = 0.0183$, indicating that the final state stabilizes for any future time (iteration) in a free-flow value, like a transient of a road being emptied. The initial condition chosen, as well as the others that will be shown in the rest of this work, has been selected due to it exhibits a rapid transient, which permits a better view of the iterations behavior.

We are including the other two fundamental diagrams with their respective iterations carried on each value of $v_o$ for the respective initial condition. The plots that correspond with velocity--density and velocity--flow diagrams show the convergence to $v = 1$ and confirm the stable point for $k$ and $q$ as calculated in Section \ref{sec:MatConc}. These plots do not repeat the same idea of including a 45-degree line to yield the iterations, but it is possible to watch that the iterations follow not only the respective fundamental diagram profile, but they follow the profile of the missing variable diagram, i.e. if the iterations are over the velocity-density fundamental diagram, iterations also touch the velocity-flow fundamental diagram profile. This is a consequence of the relation among the three macroscopic variables. 
 
If the initial condition is changed, the transient will follow a distinct trajectory, but it will reach the same final value, that is to say, the steady state of both situations are equal. It is said that the point $q=k=0.0183$ (and $v=1$, simultaneously) is an \textit{attractor} for the trajectories of the system (as described by \cite{HillbornChaos}, for example).

As the parameter $v_o$ is varied, a set of final values result, which can be plotted as in Figure \ref{kvobif}, which is called a \textit{bifurcation map}. Vertical lines are drawn on it corresponding with the values of $v_o$ used and crossing with the final value reached. Since this work takes into account the three macroscopic variables for traffic, the $v-v_o$ bifurcation map is also presented (Figure \ref{vvobif}).

Figure \ref{uo1p25} shows another set of iterations for the three fundamental diagrams when optimum velocity is adjusted to $v_o=1.25$ and the initial condition is $q=k=0.1$. Again, a stable point is reached, but unlike the first set of iterations, where free-flow values were reach, in this case a set of congested-flow values are obtained, even though the initial density was for a free-flow regime. This simulates a congestion process in the network, due to characteristics on the road represented by $v_o$. This time the achieved values are $q=k=0.4493$ and the velocity is again $v=1$. Bifurcation maps \ref{kvobif} and \ref{vvobif} exhibit this values on their own. Like in the past simulation, if initial condition is changed only the transient is different, but the steady state reached at the end becomes the same again.

We have until now two distinct values of $v_0$ and it must be watched that the first one is a value $v_0<1$ while the second one $v_0>1$. Relating with Theorem \ref{th1} it is now possible to see that values $v_0<1$ represent a road that tends to get empty, while values $v_0>1$ represent a road that tends to congestion.

\begin{figure}[t]
\centering
\includegraphics[width=1.0\textwidth]{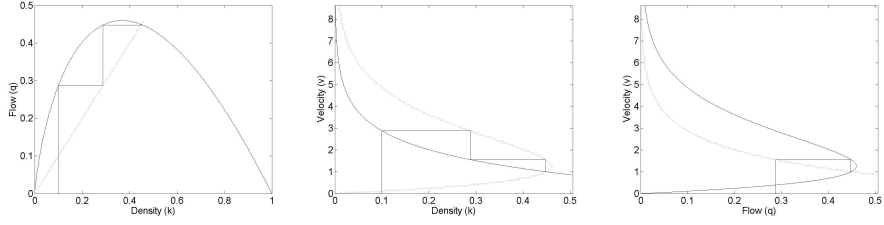}
\caption{Iterative behavior of the normalized fundamental diagrams. Optimum velocity $v_o=1.25$, initial condition $k=0.1$} \label{uo1p25}
\end{figure}
\begin{figure}[t]
\centering
\includegraphics[width=1.0\textwidth]{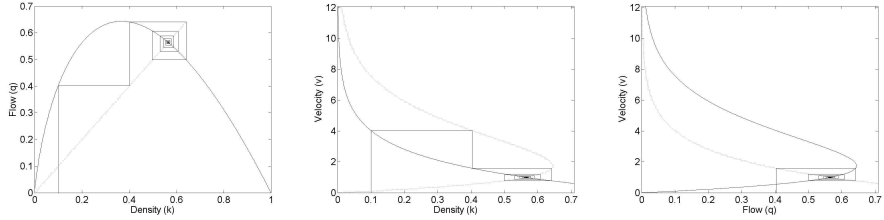}
\caption{Iterative behavior of the normalized fundamental diagrams. Optimum velocity $v_o=1.75$, initial condition $k=0.1$} \label{uo1p75}
\end{figure}
Keeping the initial condition but changing for $v_o=1.75$ the situation runs to another stable point, but the iterations become damped cyclic (Figure \ref{uo1p75}). As will be seen in brief, this behavior marks an important limit depicted by the bifurcation maps.

If $v_o= 2.25$ the attractor for traffic variables behaves purely cyclic (Figure \ref{uo2p25}). In fact, values for density and flow fluctuates among $q=k=0.3533$ and $q=k=0.8271$ and the values for velocity between $v=2.3409$ and $v=0.4272$, once the transients have vanished. This is corresponded in the bifurcation maps, where this value of optimum velocity marks
the region where two branches signal a set of pairs of final values. The physical meaning is an oscillatory response of the vehicles, which pass from near optimum values to congested-flow values from one time (iteration) to the next.

If optimum velocity is varied again to $v_o=2.405$ (Figure \ref{uo2p405}), with an initial condition in this case of $k=0.275$, traffic becomes a 4-period cycle, with values $q=k=[0.8496, 0.3330, 0.8806, 0.2692]$ and $v=[3.1560, 0.3919, 2.6446, 0.3057]$.

If attention is focused on the bifurcation maps the branches that correspond with sets of 4 values for each traffic variable runs in an interval that goes from approximately $v_o=2.395$ to $v_o=2.440$.

\begin{figure}[t]
\centering
\includegraphics[width=1.0\textwidth]{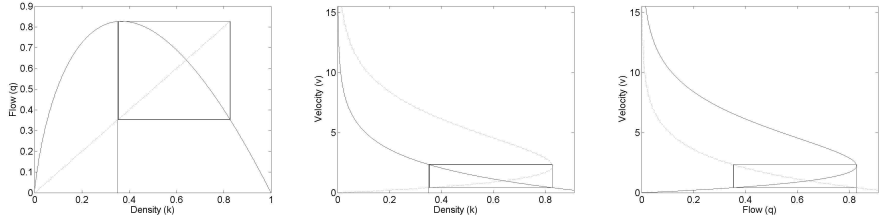}
\caption{Iterative behavior of the normalized fundamental diagrams. Optimum velocity $v_o=2.25$, initial condition $k=0.35$} \label{uo2p25}
\end{figure}
\begin{figure}
\centering
\includegraphics[width=1.0\textwidth]{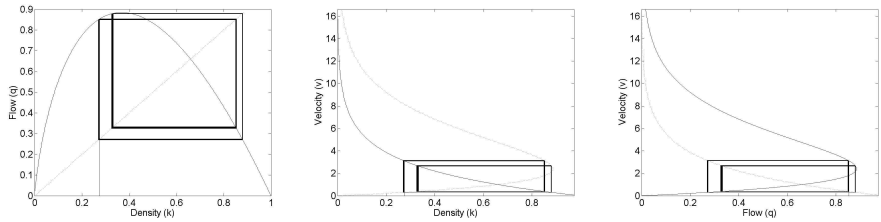}
\caption{Iterative behavior of the normalized fundamental diagrams. Optimum velocity $v_o=2.405$, initial condition $k=0.275$} \label{uo2p405}
\end{figure}
This process can continue. If now $v_o=2.48$, and initial density is $k=0.23$ the iterations are those of Figure \ref{uo2p48}, and related with the respective optimum velocity in the bifurcation maps. Cyclic behavior has become an $8-$period oscillation with $q = k = [0.8094,\, 0.4244,\, 0.9021,\, 0.2305,\, 0.8389,\, 0.3654,\, 0.9123,\, 0.2076]$ \hspace{0.5em} and \hspace{0.5em} $v=[3.8987,$\\
 $0.5243,\, 2.1256,\, 0.2555,\, 3.6392,\, 0.4356,\, 2.4965,\, 0.2276]$.

It is important to notice that the distance between two adjacent values of $v_o$ decreases for each set of bifurcations, and an $n-$period bifurcation reach more rapidly the next $2n-$period bifurcation, until they are very difficult to distinguish. This situation causes that a very little variation in the $v_o$ parameter could exhibit instabilities among trajectories. Even if the $v_o$ parameter keeps fixed, a very large number of final values can result, which are very sensitive to selection of the initial conditions.

\begin{figure}
\centering
\includegraphics[width=1.0\textwidth]{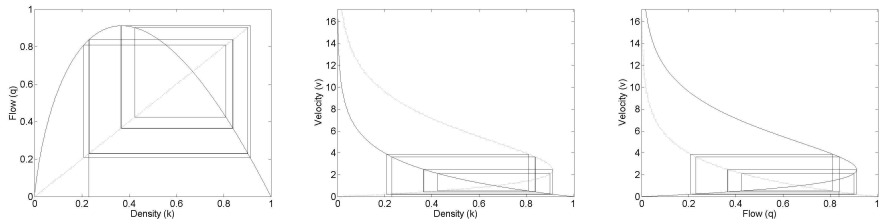}
\caption{Iterative behavior of the normalized fundamental diagrams. Optimum velocity $v_o=2.48$, initial condition $k=0.23$} \label{uo2p48}
\end{figure}
\begin{figure}
\centering
\includegraphics[width=1.0\textwidth]{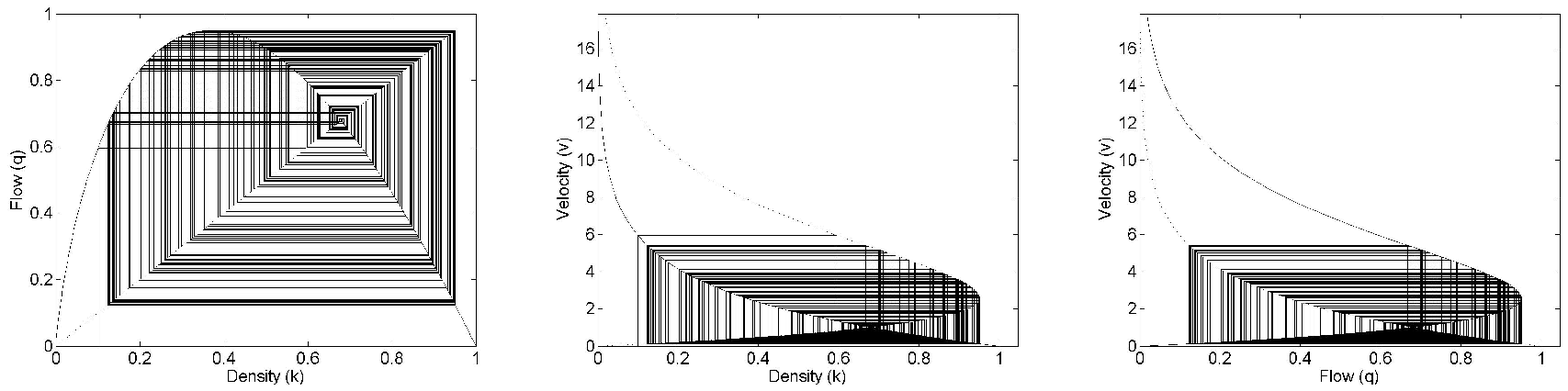}
\caption{Iterative (chaotic) behavior of the normalized fundamental diagrams. Optimum velocity $v_o=2.585$. Initial condition: $k=0.1$}
\label{uo2p585k0p1}
\end{figure}
\begin{figure}
\centering
\includegraphics[width=1.0\textwidth]{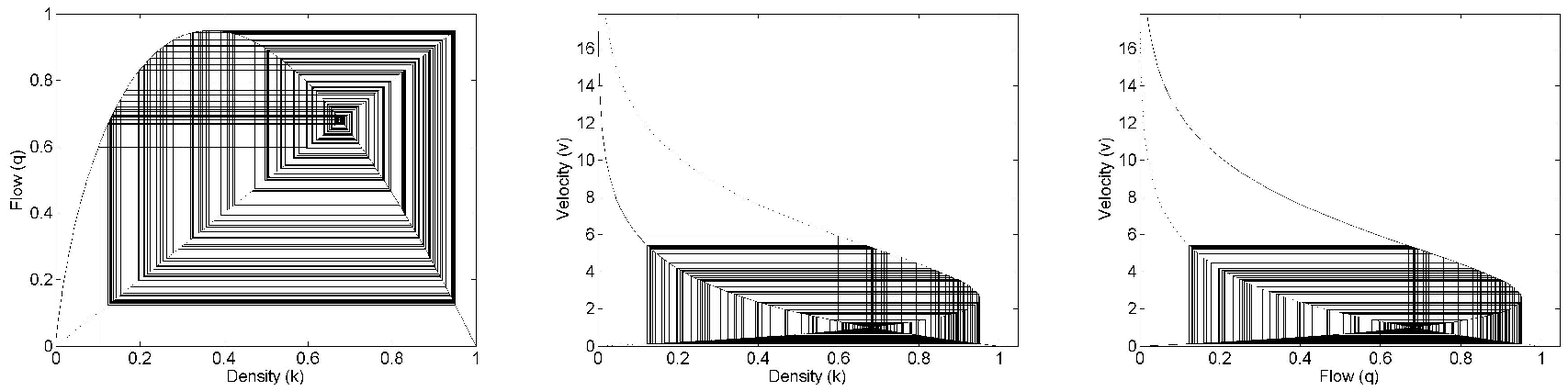}
\caption{Iterative (chaotic) behavior of the normalized fundamental diagrams. Optimum velocity $v_o=2.585$. Initial condition: $k=0.101$}
\label{uo2p585k0p101}
\end{figure}

\begin{figure}[!h]
\centering
\includegraphics[width=1.0\textwidth]{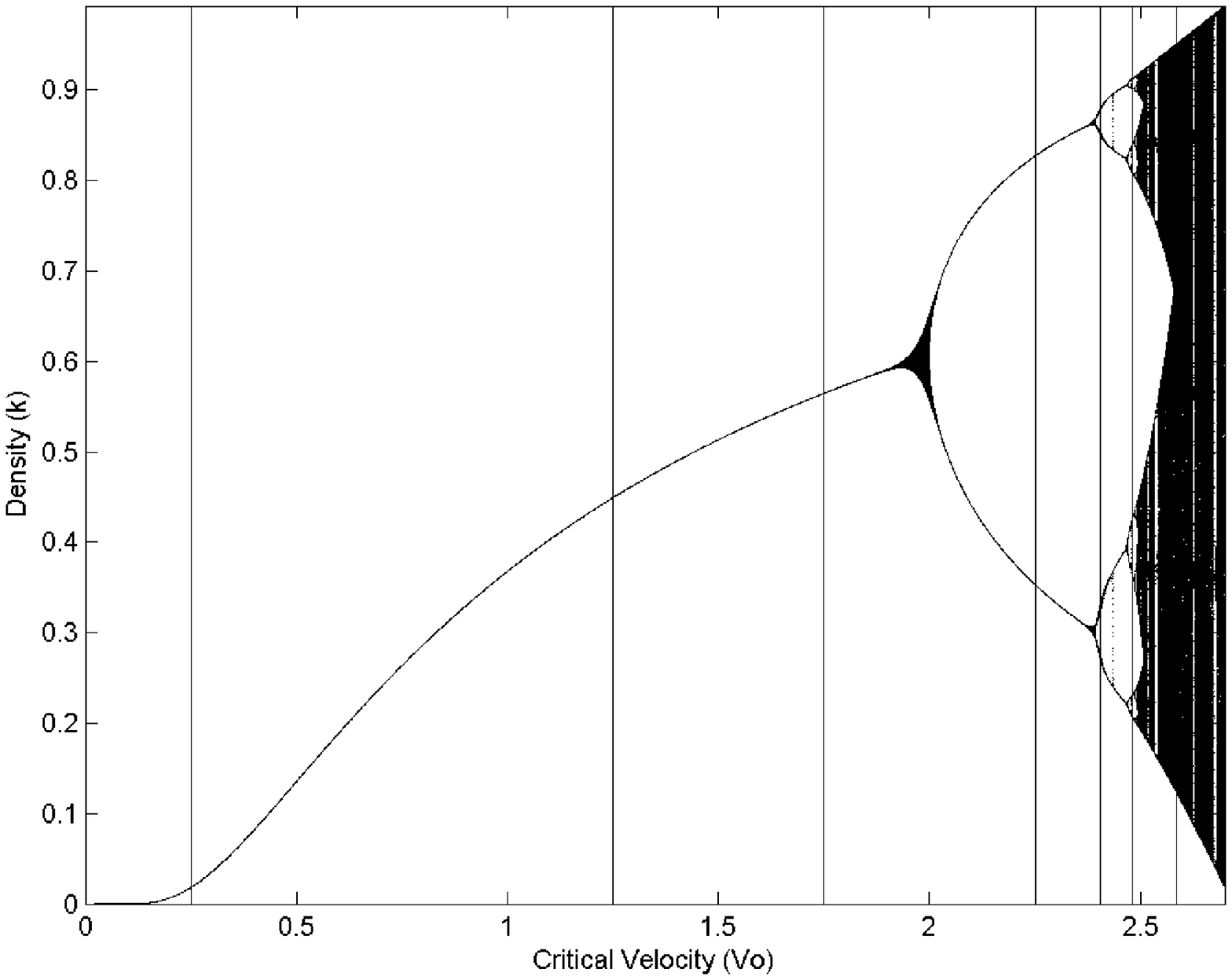}
\caption{$k-v_o$ bifurcation diagram}
\label{kvobif}
\end{figure}

An example of this situation is shown in Figure \ref{uo2p585k0p1} and Figure \ref{uo2p585k0p101}, where $v_o=2.585$ and $k=0.1$ for the first of them, which shows what it is called as \textit{chaotic trajectories}. The cyclic behavior is only apparent, and strictly speaking it does not exist because the different values do not exhibit a period where series of them can repeat themselves. This attractor is known as \textit{strange}.

The final values reached after 300 iterations are $q=k=0.8451$ and $v=4.0519$. If optimum velocity $v_o$ is left with no changes but the initial condition is slightly change to $q=k=0.101$, Figure \ref{uo2p585k0p101} depicts an apparent equal image than the last one, but a second sight will reveal that the trajectory is quite different. In fact, the final values after the same number of iterations now are $q=k=0.2020$ and $v=0.2199$.

This feature that consists in a crescent divergence between two trajectories that has very close initial states is referred as a deterministic but unpredictable behavior, proper of most nonlinear systems under certain set of parameter values. In a more linear set of trajectories, it is expected that for a small difference in initial conditions, the trajectories generated were nearby enough to not take into account that small difference or at least they are proportional. That is to say, for very close initial states, trajectories remain closely enough to achieved predictions about the system.

\begin{figure}[!h]
\centering
\includegraphics[width=1.0\textwidth]{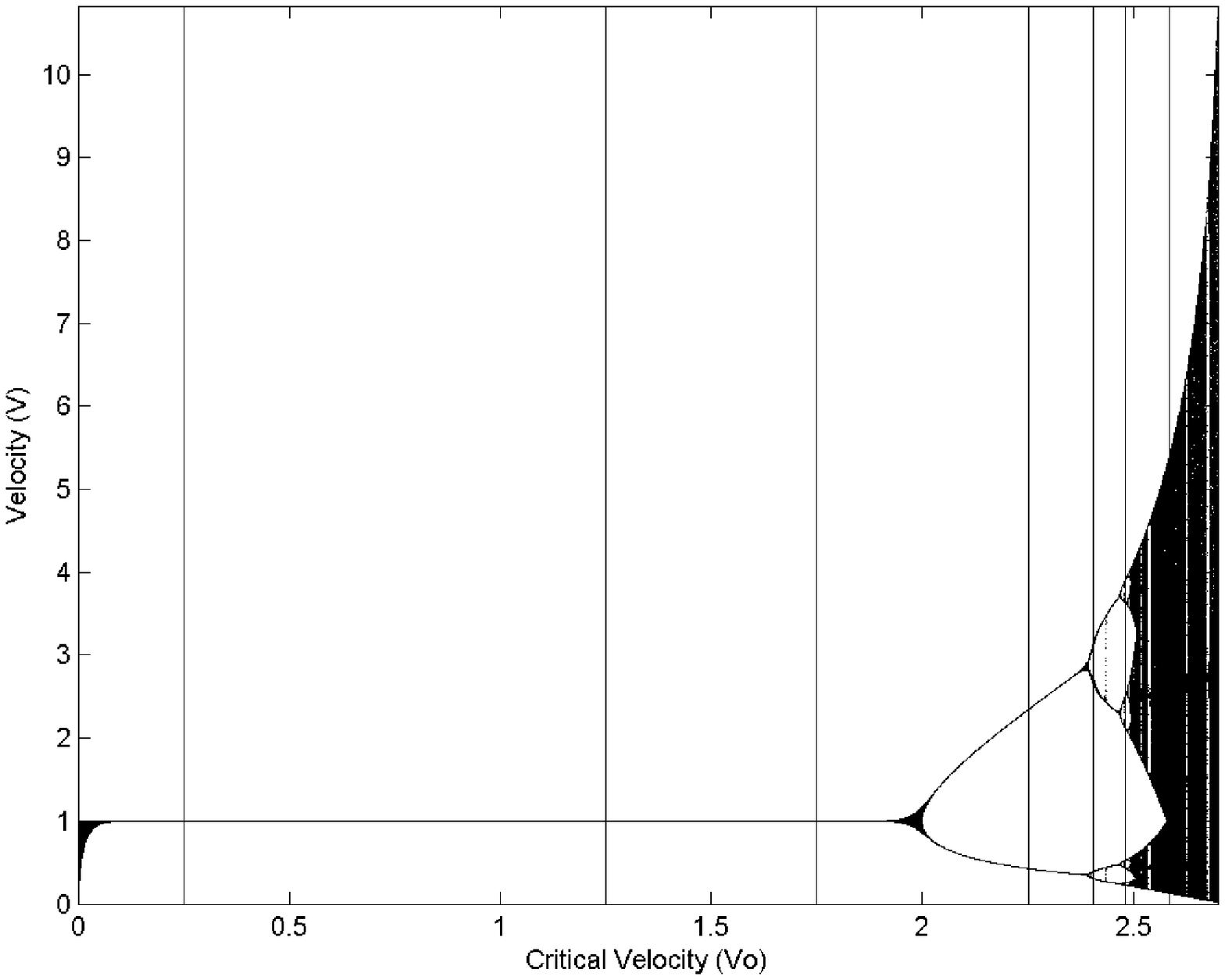}
\caption{$v-v_o$ bifurcation diagram}
\label{vvobif}
\end{figure}

However, as has been seen by the last iterations here included, this difference becomes important for a so-called chaotic system, due to in these cases a set of very close states will diverge. Physically, this means that it is impossible to perform a simulation of such a system with the intention of predicting future states, due to the real states in the beginning of such a simulations cannot be measure and transcribed to our calculations with infinite precision. Those infinitesimal differences among the measured and the real values will grow up rapidly, in an exponentially way as time (iterations) passes. Predictions will only be valid, under a tolerance, for a few future states.

\section{\label{sec:LyapExp}Analysis of Lyapunov Exponents}
Lyapunov exponents are considered a helpful tool to qualify stability of a nonlinear system (see \cite{Holmgren1994}). They are also useful to identify the chaotic feature of those kind of systems.

\begin{figure}[h]
\centering
\includegraphics[width=1.0\textwidth]{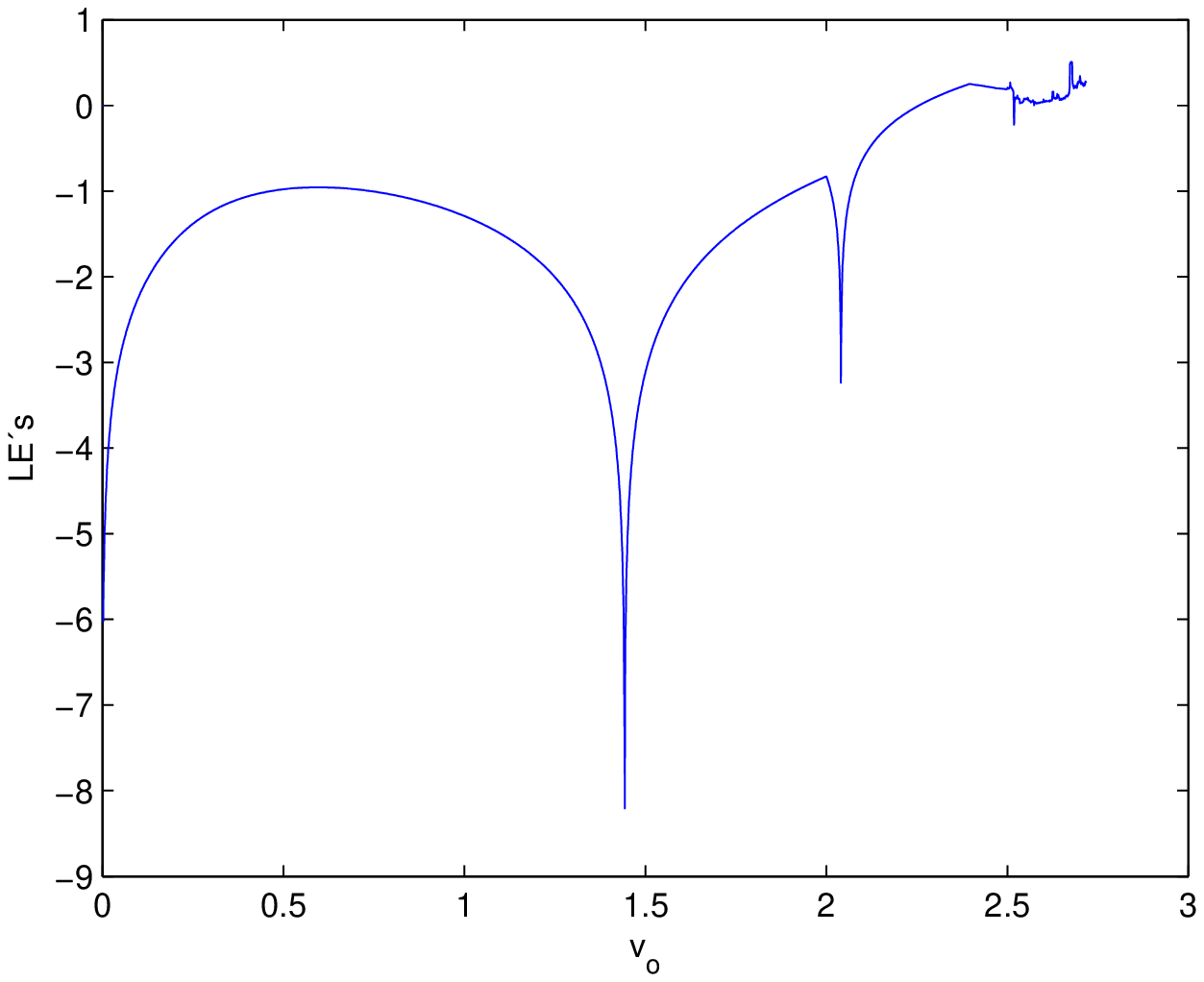}
\caption{Lyapunov exponent calculation}
\label{explyap1}
\end{figure}

\begin{figure}[h]
\centering
\includegraphics[width=1.0\textwidth]{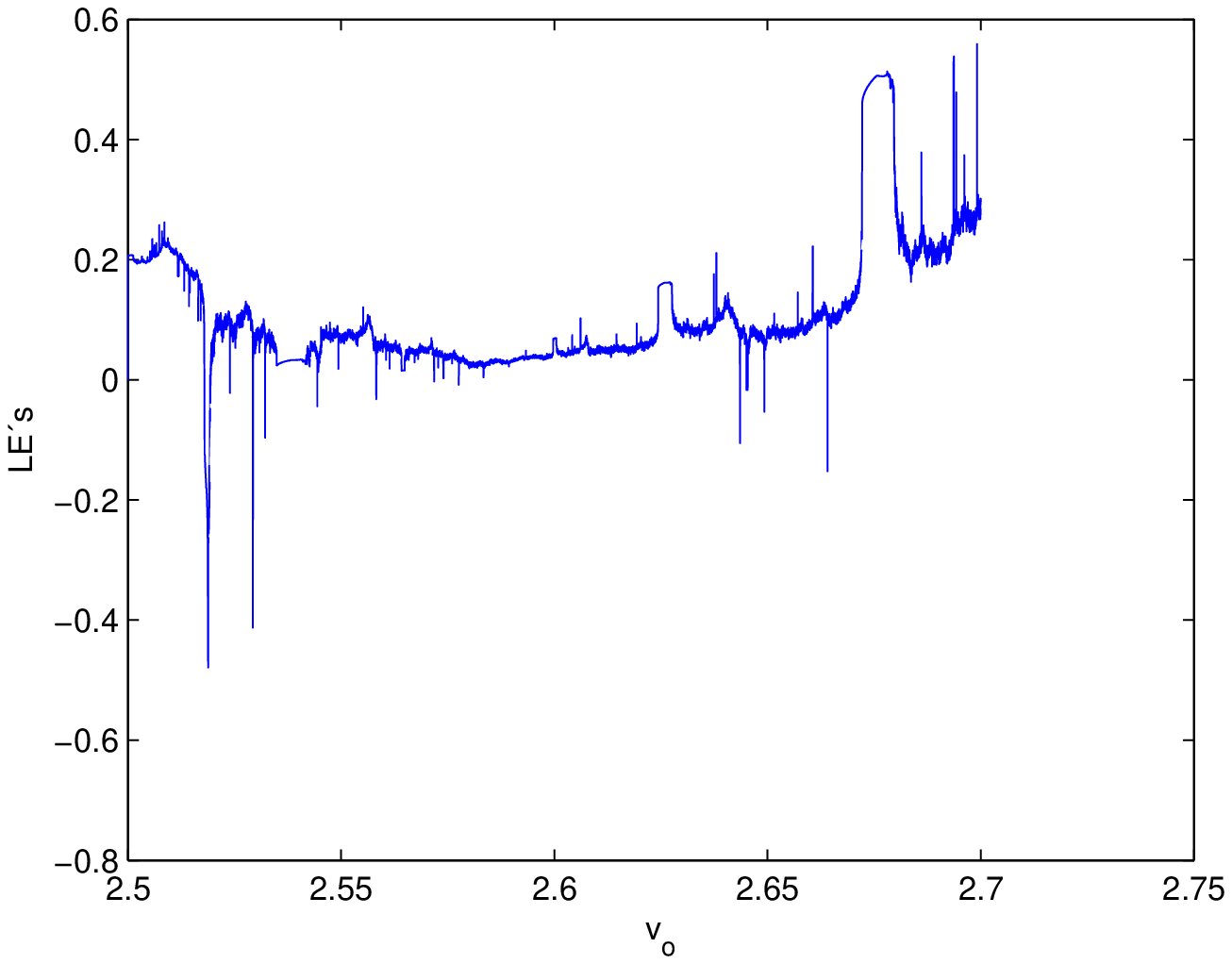}
\caption{Lyapunov exponent calculation. Zoom on $[2.5, \, 3.7]$}
\label{explyap2}
\end{figure}

Equation (\ref{LyapExpCalc}) is applied for such a calculation, which in this case $n=$ 10 000 terms.  When applying such a concept to an equation like (\ref{kiter}) it is possible to confirm the behaviors analyzed and observed in Sections \ref{sec:MatConc} and \ref{sec:Sims}.

In that way, the stable trajectories shown in Figures \ref{uo0p25}--\ref{uo1p75} that coincide with the curve that grows up from $v_0=0$ to $v_0=2.0$ in the mapping plots \ref{uo2p585k0p1} and \ref{uo2p585k0p101}, are confirmed by the negative values calculated for the Lyapunov exponents in Figure \ref{explyap1}.

There is a transition from that stable region starting on $v_0=2.0$ and ending in $v_0=2.5$ represented by the examples represented in the limit cycles depicted by Figures \ref{uo2p25}--\ref{uo2p48}. Lyapunov exponents are still negative in this interval.

However, from $v_0=2.5$ these exponents modify their sign. From this point, trajectories became chaotic, as depicted by Figures \ref{uo2p585k0p1} and \ref{uo2p585k0p101}, bifurcations in Figures \ref{kvobif} and \ref{vvobif} have become indistinguisable and Lyapunov exponents turn positive. Figure \ref{explyap2} is a magnification of the region for $v_o \geq 2.5$. 

These values and trajectories have another important feature. They are different from those analogous values obtained for a logistic mapping and are exclusive for the solution of Greenberg's model (\ref{kiter}).

\section{Concluding remarks} 
There are different models that deal with the relationship among the macroscopic variables density $k$, flow $q$ and velocity $v$, like Greenberg's model, which is well fitted for real data, and even though it is not well suited for free-flow velocities, it has been proved in a wide range of practical values, and its function solution is easy to manipulate to perform proper analyses of different behaviors exhibited as a parameter included in it is varied.

A normalization is carried on the variables density $k$ and flow $q$ in order to generalize results. The parameter that is varied in Greenberg's model is the critical velocity $v_0$, which modifies the scale of the fundamental diagrams that can be produced for that expression. This is equivalent to modify the conditions on a road network represented by those fundamental diagrams, giving as a consequence different behaviors in the modelled traffic.

One of the main contributions of these article is to include the fundamental diagrams of the flow--velocity and the velocity--density plots and to perform the respective iterations and analysis on them, instead of using solely the flow--density fundamental diagram as if it was a logistic map.

Iterations on values of density $k$ substituted in Greenberg's model to obtain flow $q$ are performed, and by the normalization made those values are directly taken as new values of density, which are again substituted to get new ones. This set of values are plotted with the help of a straight line $q=k$, emulating the evolution in time of those traffic values. In a similar and parallel manner, values of velicty $v$ are also calculated and plotted. 

When $v_o$ is changed from low to higher values, iterations draw trajectories that converge to stable, cyclic or chaotic values. It is possible to obtain a (bifurcation) map of final values for each $v_o$ modification. It is easy to recognized in this plot which ranges of the parameter $v_o$ give stable, cyclic or chaotic behavior. Another main contribution of this article is to show the velocity--optimum velocity bifurcation map obtained from the iteration schemes here performed.

A main feature of the bifurcations is that they split in $2^n$-period cycles, i. e. for the first time a bifurcation map splits in two branches, then it divides in four, latter in eight, sixteen and so on. But this is also done in shorter intervals and cannot continue for ever. Soon, the period $2^n$ is juxtaposed to the next $2^{n+1}$ and any little variation in the parameter $v_o$ or in the initial conditions will give an unstable situation between branches, which implies big differences between two final states for very close initial values of those states. This sensitivity to initial conditions is a well known feature of chaotic systems, like the cases here presented.

Chaos is mainly characterized by the sensitivity of the system to initial conditions, i.e. two trajectories starting from very nearby initial values diverge from each other. A measure of this divergence is the Lyapunov exponents, which are negative for stable and cyclic trajectories, but are positive for chaotic ones.

Chaotic behavior is mainly characterized by the sensitivity of the system to initial conditions, that is, two trajectories of variables that start from very nearby initial values diverge from each other. A measure of this divergence is the Lyapunov exponents, which are negative for stable and cyclic trajectories, but is positive for chaotic ones. Further analyses of the limit cycles and chaotic regions can be carried on since a point of view of Lyapunov stability, extending the study of this new approach in further work.

\end{document}